\documentclass[a4paper,12pt]{article}
\usepackage[utf8]{inputenc}
\usepackage[T1]{fontenc}
\usepackage{amsfonts,amsmath,amssymb,amsthm}
\theoremstyle{definition}

\theoremstyle{plain}

\newtheorem{proposition}{Proposition}
\newtheorem{theorem}{Theorem}
\usepackage{graphicx}
\usepackage{caption}
\usepackage{pgf,tikz,circuitikz,subfig}
\usetikzlibrary{arrows,shapes,positioning}
\usepackage{booktabs,multirow,tabularx}
\usepackage{hyperref}
\hypersetup{colorlinks,%
	citecolor={red}, 
	urlcolor={blue},
	linkcolor={blue},
	breaklinks={true},
	pagebackref={true},
	hyperindex={true},
}
\usepackage{url}
\usepackage{cite}
\usepackage[left=2cm,right=2cm,top=2cm,bottom=2cm]{geometry}
\usepackage[ruled]{algorithm}
\usepackage{algorithmic}
\usepackage{newtxtext}
\usepackage{courier}
\usepackage{enumitem}
\newlist{abbrv}{itemize}{1}
\setlist[abbrv,1]{label=,labelwidth=0.9in,align=parleft,noitemsep,leftmargin=!}

\newcommand{\ub}[1]{\overline{#1}}

\newcommand{\geo}[1]{\mathtt{#1}}

\title{Tight bounds on the maximal perimeter of convex equilateral small polygons}
\author{Christian Bingane\thanks{Department of Mathematics and Industrial Engineering, Polytechnique Montreal, Montreal, Quebec, Canada, H3C 3A7. Emails: \url{christian.bingane@polymtl.ca}, \url{charles.audet@polymtl.ca}} \and Charles Audet\footnotemark[1]}
\begin{document}
\maketitle
\begin{abstract}
A small polygon is a polygon that has diameter one. The maximal perimeter of a convex equilateral small polygon with $n=2^s$ sides is not known when $s \ge 4$. In this paper, we construct a family of convex equilateral small $n$-gons, $n=2^s$ and $s \ge 4$, and show that their perimeters are within $O(1/n^4)$ of the maximal perimeter and exceed the previously best known values from the literature. In particular, for the first open case $n=16$, our result proves that Mossinghoff's equilateral hexadecagon is suboptimal.
\end{abstract}


\section{Introduction}
The {\em diameter} of a polygon is the largest distance between any pair of its vertices. A polygon is called {\em small} if its diameter equals one. We recall that an {\em equilateral} polygon is a polygon that has all sides of the same length and a {\em regular} polygon is an equilateral polygon whose interior angles are equal. For an integer $n \ge 3$, the problem of finding the maximal perimeter of a convex small $n$-gon was investigated by Reinhardt~\cite{reinhardt1922} in 1922, Vincze~\cite{vincze1950} in 1950, and Datta~\cite{datta1997} in 1997. They proved that for $n \ge 3$
\begin{itemize}
	\item the value $2n\sin \frac{\pi}{2n}$ is an upper bound on the perimeter of any convex small $n$-gon;
	\item the regular small $n$-gon is an optimal solution only when $n$ is odd;
	\item there are finitely many optimal solutions~\cite{mossinghoff2011,hare2013,hare2019} when $n$ has an odd factor and these solutions are all equilateral.
\end{itemize}

When $n$ is a power of $2$, the maximal perimeter problem is solved for $n = 4$ and $n=8$. The case $n=4$ was solved by Tamvakis~\cite{tamvakis1987} in 1987 and the case $n=8$ by Audet, Hansen, and Messine~\cite{audet2007a} in 2007. Both optimal $4$-gon and $8$-gon, shown respectively in Figure~\ref{figure:4gon:Q4} and Figure~\ref{figure:8gon:V8}, are not equilateral. For $n=2^s$ with integer $s\ge 4$, exact solutions in the maximal perimeter problem appear to be presently out of reach. However, tight lower bounds may be obtained analytically. Recently, Bingane~\cite{bingane2022e,bingane2021d} constructed a family of convex non-equilateral small $n$-gons, for $n=2^s$ with $s\ge 4$, and proved that the perimeters obtained cannot be improved for large $n$ by more than $\pi^9/(8n^8)$.

The diameter graph of a polygon is the graph with the vertices of the polygon, and an edge between two vertices exists only if the distance between these vertices equals the diameter. Figure~\ref{figure:4gon}, Figure~\ref{figure:6gon}, and Figure~\ref{figure:8gon} represent diameter graphs of some convex small polygons. The solid lines illustrate pairs of vertices which are unit distance apart. Vincze~\cite{vincze1950} studied the problem of finding the minimal diameter of a convex polygon with unit-length sides. This problem is equivalent to the equilateral case of the maximal perimeter problem. He showed that a necessary condition of a convex equilateral small polygon to have maximal perimeter is that each vertex should have an opposite vertex at a distance equal to the diameter. It is easy to see that for $n=4$, the maximal perimeter of a convex equilateral small $4$-gon is only attained by the regular $4$-gon. Vincze also described a convex equilateral small $8$-gon, shown in Figure~\ref{figure:8gon:X8}, with longer perimeter than that of the regular $8$-gon. In 2004, Audet, Hansen, Messine, and Perron~\cite{audet2004} used both geometrical arguments and methods of global optimization to determine the unique convex equilateral small $8$-gon with the longest perimeter, illustrated in Figure~\ref{figure:8gon:H8}.

For $n=2^s$ with integer $s\ge 4$, the equilateral case of the maximal perimeter problem remains unsolved and, as in the general case, exact solutions appear to be presently out of reach. In 2008, Mossinghoff~\cite{mossinghoff2008} constructed a family of convex equilateral small $n$-gons, for $n=2^s$ with $s\ge 4$, whose perimeters differ from the upper bound $2n\sin \frac{\pi}{2n}$ by just $O(1/n^4)$. By contrast, the perimeter of a regular small $n$-gon differs by $O(1/n^2)$ when $n \ge 4$ is even. In the present paper, we propose tighter lower bounds on the maximal perimeter of convex equilateral small $n$-gons when $n=2^s$ and integer $s \ge 4$ by a constructive approach. Thus, our main result is the following:

\begin{theorem}\label{thm:Bn}
	Suppose $n=2^s$ with integer $s\ge 4$. Let $\ub{L}_n := 2n \sin \frac{\pi}{2n}$ denote an upper bound on the perimeter $L(\geo{P}_n)$ of a convex small $n$-gon $\geo{P}_n$. Let $\geo{M}_n$ denote the convex equilateral small $n$-gon constructed by Mossinghoff~\cite{mossinghoff2008}. Then there exists a convex equilateral small $n$-gon $\geo{A}_n$ such that
	\[
	\ub{L}_n - L(\geo{A}_n) = \frac{\pi^4}{n^4} + O\left(\frac{1}{n^5}\right)
	\]
	and
	\[
	L(\geo{A}_n) - L(\geo{M}_n) = \frac{2\pi^4}{n^4} + O\left(\frac{1}{n^5}\right) > 0.
	\]
\end{theorem}

In addition, we show that the resulting polygons for $n=32$ and $n=64$ are not optimal by providing two convex equilateral small polygons with longer perimeters.

The remainder of this paper is organized as follows. Section~\ref{sec:ngon} recalls principal results on the maximal perimeter of convex small polygons. Section~\ref{sec:Bn} considers the polygons $\geo{A}_n$ and shows that they satisfy Theorem~\ref{thm:Bn}. Section~\ref{sec:optimal} shows that the polygons $\geo{A}_{32}$ and $\geo{A}_{64}$ are not optimal by constructing a $32$-gon and a $64$-gon with larger perimeters. Concluding remarks are presented in Section~\ref{sec:conclusion}.

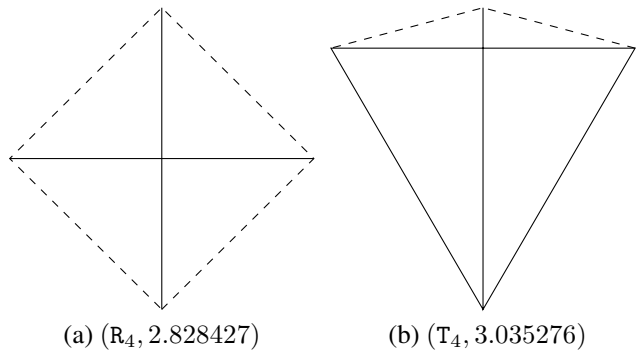
\begin{figure}[h]
	\centering
	\subfloat[$(\geo{R}_4,2.828427)$]{
		\begin{tikzpicture}[scale=4]
			\draw[dashed] (0,0) -- (0.5000,0.5000) -- (0,1) -- (-0.5000,0.5000) -- cycle;
			\draw (0,0) -- (0,1);
			\draw (0.5000,0.5000) -- (-0.5000,0.5000);
		\end{tikzpicture}
	}
	\subfloat[$(\geo{T}_4,3.035276)$]{
		\begin{tikzpicture}[scale=4]
			\draw[dashed] (0.5000,0.8660) -- (0,1) -- (-0.5000,0.8660);
			\draw (0,1) -- (0,0) -- (0.5000,0.8660) -- (-0.5000,0.8660) -- (0,0);
		\end{tikzpicture}
		\label{figure:4gon:Q4}
	}
	\caption{Two convex small $4$-gons $(\geo{P}_4,L(\geo{P}_4))$: (a) Regular $4$-gon; (b) Tamvakis $4$-gon~\cite{tamvakis1987}}
	\label{figure:4gon}
\end{figure}

\begin{figure}[h]
	\centering
	\subfloat[$(\geo{R}_6,3)$]{
		\begin{tikzpicture}[scale=4]
		\draw[dashed] (0,0) -- (0.4330,0.2500) -- (0.4330,0.7500) -- (0,1) -- (-0.4330,0.7500) -- (-0.4330,0.2500) -- cycle;
		\draw (0,0) -- (0,1);
		\draw (0.4330,0.2500) -- (-0.4330,0.7500);
		\draw (0.4330,0.7500) -- (-0.4330,0.2500);
		\end{tikzpicture}
	}
	\subfloat[$(\geo{R}_{3,6},3.105829)$]{
		\begin{tikzpicture}[scale=4]
		\draw[dashed] (0,0) -- (0.3660,0.3660) -- (0.5000,0.8660) -- (0,1) -- (-0.5000,0.8660) -- (-0.3660,0.3660) -- cycle;
		\draw (0,0) -- (0.5000,0.8660) -- (-0.5000,0.8660) -- cycle;
		\draw (0,0) -- (0,1);
		\draw (0.3660,0.3660) -- (-0.5000,0.8660);
		\draw (0.5000,0.8660) -- (-0.3660,0.3660);
		\end{tikzpicture}
	\label{figure:6gon:R36}
	}
	\caption{Two convex equilateral small $6$-gons $(\geo{P}_6,L(\geo{P}_6))$: (a) Regular $6$-gon; (b) Reinhardt $6$-gon~\cite{reinhardt1922}}
	\label{figure:6gon}
\end{figure}
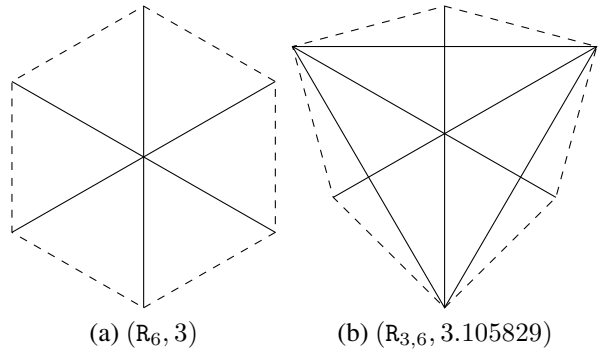

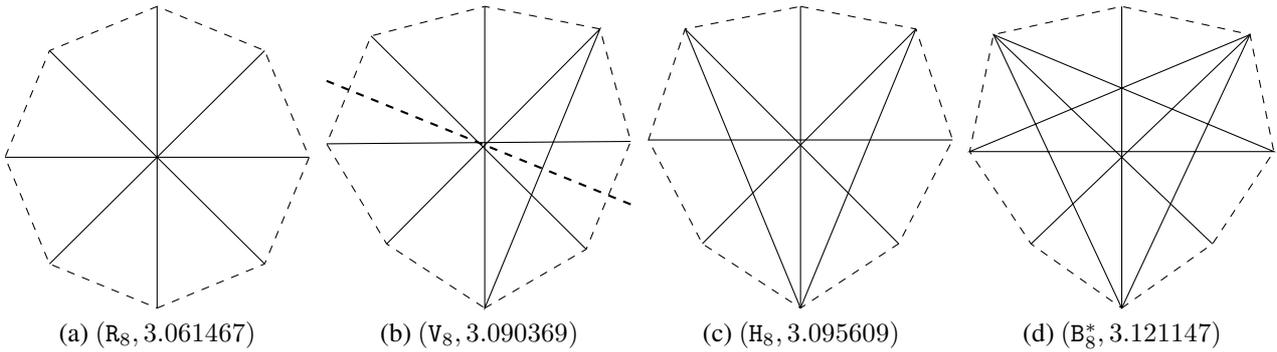
\begin{figure}[h]
	\centering
	\subfloat[$(\geo{R}_8,3.061467)$]{
		\begin{tikzpicture}[scale=4]
		\draw[dashed] (0,0) -- (0.3536,0.1464) -- (0.5000,0.5000) -- (0.3536,0.8536) -- (0,1) -- (-0.3536,0.8536) -- (-0.5000,0.5000) -- (-0.3536,0.1464) -- cycle;
		\draw (0,0) -- (0,1);
		\draw (0.3536,0.1464) -- (-0.3536,0.8536);
		\draw (0.5000,0.5000) -- (-0.5000,0.5000);
		\draw (0.3536,0.8536) -- (-0.3536,0.1464);
		\end{tikzpicture}
	}
	\subfloat[$(\geo{V}_8,3.090369)$]{
		\begin{tikzpicture}[scale=4]
		\draw[dashed] (0,0) -- (0.3335,0.1950) -- (0.4799,0.5525) -- (0.3790,0.9254) -- (0,1) -- (-0.3737,0.9021) -- (-0.5201,0.5446) -- (-0.3225,0.2127) -- cycle;
		\draw (0.3335,0.1950) -- (-0.3737,0.9021);\draw (0.4799,0.5525) -- (-0.5201,0.5446);
		\draw (0,1) -- (0,0) -- (0.3790,0.9254) -- (-0.3225,0.2127);
		\draw[thick,dashed] (0.4799,0.3438) -- (-0.5201,0.7533);
		\end{tikzpicture}
	\label{figure:8gon:X8}
	}
\subfloat[$(\geo{H}_8,3.095609)$]{
	\begin{tikzpicture}[scale=4]
		\draw[dashed] (0,1) -- (0.3796,0.9251) -- (0.5000,0.5574) -- (0.3228,0.2134) -- (0,0) -- (-0.3228,0.2134) -- (-0.5000,0.5574) -- (-0.3796,0.9251) -- cycle;
		\draw (0,0) -- (0,1);
		\draw (0,0) -- (0.3796,0.9251);\draw (0,0) -- (-0.3796,0.9251);
		\draw (0.3796,0.9251) -- (-0.3228,0.2134);\draw (-0.3796,0.9251) -- (0.3228,0.2134);
		\draw (0.5000,0.5574) -- (-0.5000,0.5574);
	\end{tikzpicture}
		\label{figure:8gon:H8}
}
\subfloat[$(\geo{B}_8^*,3.121147)$]{
	\begin{tikzpicture}[scale=4]
		\draw[dashed] (0,0) -- (0.2983,0.2128) -- (0.5000,0.5188) -- (0.4217,0.9067) -- (0,1) -- (-0.4217,0.9067) -- (-0.5000,0.5188) -- (-0.2983,0.2128) -- cycle;
		\draw (0,0) -- (0,1);
		\draw (0,0) -- (0.4217,0.9067) -- (-0.5000,0.5188) -- (0.5000,0.5188)-- (-0.4217,0.9067) -- cycle;
		\draw (0.4217,0.9067) -- (-0.2983,0.2128);\draw (-0.4217,0.9067) -- (0.2983,0.2128);
	\end{tikzpicture}
	\label{figure:8gon:V8}
}
\caption{Four convex small $8$-gons $(\geo{P}_8,L(\geo{P}_8))$: (a) Regular $8$-gon; (b) Vincze $8$-gon~\cite{vincze1950}; (c) Optimal equilateral $8$-gon~\cite{audet2004}; (d) Optimal non-equilateral $8$-gon~\cite{audet2007a}}
\label{figure:8gon}
\end{figure}


\section{Perimeters of convex equilateral small polygons}\label{sec:ngon}
Let $L(\geo{P})$ denote the perimeter of a polygon $\geo{P}$. For a given integer $n\ge 3$, let $\geo{R}_n$ denote the regular small $n$-gon. We have
\[
L(\geo{R}_n) =
\begin{cases}
	2n\sin \frac{\pi}{2n} &\text{if $n$ is odd,}\\
	n\sin \frac{\pi}{n} &\text{if $n$ is even.}\\
\end{cases}
\]

When $n$ has an odd factor $m$, consider the family of convex equilateral small $n$-gons constructed as follows:
\begin{enumerate}
	\item Transform the regular small $m$-gon  $\geo{R}_m$ into a Reuleaux $m$-gon by replacing each edge by a circle's arc passing through its end vertices and centered at the opposite vertex;
	\item Add at regular intervals $n/m-1$ vertices within each arc;
	\item Take the convex hull of all vertices.
\end{enumerate}
These $n$-gons are denoted $\geo{R}_{m,n}$ and $L(\geo{R}_{m,n}) = 2n\sin \frac{\pi}{2n}$. The $6$-gon $\geo{R}_{3,6}$ is illustrated in Figure~\ref{figure:6gon:R36}.

\begin{theorem}[Reinhardt~\cite{reinhardt1922}, Vincze~\cite{vincze1950}, Datta~\cite{datta1997}]\label{thm:perimeter}
	For all $n \ge 3$, let $L_n^*$ denote the maximal perimeter among all convex small $n$-gons, $\ell_n^*$ the maximal perimeter among all equilateral ones, and $\ub{L}_n := 2n \sin \frac{\pi}{2n}$.
	\begin{itemize}
		\item When $n$ has an odd factor $m$, $\ell_n^* = L_n^* = \ub{L}_n$ is achieved by finitely many equilateral $n$-gons~\cite{mossinghoff2011,hare2013,hare2019}, including~$\geo{R}_{m,n}$. The optimal $n$-gon $\geo{R}_{m,n}$ is unique if $m$ is prime and $n/m \le 2$.
		\item When $n=2^s$ with $s\ge 2$, $L(\geo{R}_n) < L_n^* < \ub{L}_n$.
	\end{itemize}
\end{theorem}

When $n=2^s$, both $L_n^*$ and $\ell_n^*$ are only known for $s \le 3$. Tamvakis~\cite{tamvakis1987} found that $L_4^* = 2+\sqrt{6}-\sqrt{2}$, and this value is only achieved by $\geo{T}_4$, represented in Figure~\ref{figure:4gon:Q4}. Audet, Hansen, and Messine~\cite{audet2007a} proved that $L_8^* = 3.121147\dots$, and this value is only achieved by $\geo{B}_8^*$, represented in Figure~\ref{figure:8gon:V8}. For the equilateral quadrilateral, it is easy to see that $\ell_4^* = L(\geo{R}_4) = 2\sqrt{2}$. Audet, Hansen, Messine and Perron~\cite{audet2004} studied the equilateral octagon and determined that $\ell_8^* = 3.095609\ldots > L(\geo{R}_8) = 4\sqrt{2-\sqrt{2}}$, and this value is only achieved by $\geo{H}_8$, represented in Figure~\ref{figure:8gon:H8}. If $u := {\ell_8^*}^2/64$ denotes the square of the sides length of $\geo{H}_8$, we can show that $u$ is the unique root of the polynomial equation
\[
2u^6 - 18u^5 + 57u^4 -78u^3+46u^2-12u+1=0
\]
that belongs to $(\sin^2(\pi/8),4\sin^2(\pi/16))$. Note that the following inequalities are strict: $\ell_4^* < L_4^*$ and $\ell_8^* < L_8^*$.

For $n=2^s$ with $s\ge 4$, exact solutions of the maximal perimeter problem appear to be presently out of reach. However, tight lower bounds may be obtained analytically. Recently, Bingane~\cite{bingane2022e,bingane2021d} proved that, for $n=2^s$ with $s\ge 4$,
\[
L_n^* \ge 2n \sin \frac{\pi}{2n} \cos \left(\frac{1}{2}\arctan \left(\tan \frac{2\pi}{n}\tan \frac{\pi}{n}\right)-\frac{1}{2}\arcsin\left(\frac{\sin (2\pi/n) \sin (\pi/n)}{\sqrt{4\sin^2(\pi/n) + \cos (4\pi/n)}}\right)\right),
\]
which implies
\[
\ub{L}_n - L_n^* \le \frac{\pi^9}{8n^8} + O\left(\frac{1}{n^{10}}\right).
\]
On the other hand, Mossinghoff~\cite{mossinghoff2008} constructed a family of convex equilateral small $n$-gons $\geo{M}_n$, illustrated in Figure~\ref{figure:Mn}, such that
\[
\ub{L}_n - L(\geo{M}_n) = \frac{3\pi^4}{n^4} + O\left(\frac{1}{n^5}\right)
\]
and
\[
L(\geo{M}_n) - L(\geo{R}_n) = \frac{\pi^3}{8n^2} + O\left(\frac{1}{n^4}\right) > 0
\]
for $n=2^s$ with $s\ge 4$. The next section proposes tighter lower bounds for $\ell_n^*$.

\begin{figure}
	\centering
	\subfloat[$(\geo{M}_{16},3.134707)$]{
		\begin{tikzpicture}[scale=6]
			\draw[dashed] (0,0) -- (0.1875,0.0568) -- (0.3390,0.1811) -- (0.4315,0.3538) -- (0.4885,0.5412) -- (0.4922,0.7311) -- (0.3678,0.8885) -- (0.1950,0.9808) -- (0,1) -- (-0.1950,0.9808) -- (-0.3678,0.8885) -- (-0.4922,0.7311) -- (-0.4885,0.5412) -- (-0.4315,0.3538) -- (-0.3390,0.1811) -- (-0.1875,0.0568) -- cycle;
			\draw[red,thick] (0,0)--(0,1);
			\draw[blue,thick] (0,0) -- (0.1950,0.9808) -- (-0.1875,0.0568) -- (0.3678,0.8885) -- (-0.3390,0.1811) -- (0.4922,0.7311) -- (-0.4885,0.5412);
			\draw[blue,thick] (0,0) -- (-0.1950,0.9808) -- (0.1875,0.0568) -- (-0.3678,0.8885) -- (0.3390,0.1811) -- (-0.4922,0.7311) -- (0.4885,0.5412);
			\draw (0.4922,0.7311) -- (-0.4315,0.3538);\draw (-0.4922,0.7311) -- (0.4315,0.3538);
		\end{tikzpicture}
	}
	\subfloat[$(\geo{M}_{32},3.140134)$]{
		\begin{tikzpicture}[scale=6]
			\draw[dashed] (0,0) -- (0.0971,0.0144) -- (0.1895,0.0475) -- (0.2736,0.0979) -- (0.3525,0.1564) -- (0.4184,0.2291) -- (0.4603,0.3178) -- (0.4842,0.4129) -- (0.4986,0.5100) -- (0.4966,0.6081) -- (0.4635,0.7005) -- (0.4131,0.7847) -- (0.3546,0.8635) -- (0.2819,0.9294) -- (0.1932,0.9713) -- (0.0980,0.9952) -- (0,1) -- (-0.0980,0.9952) -- (-0.1932,0.9713) -- (-0.2819,0.9294) -- (-0.3546,0.8635) -- (-0.4131,0.7847) -- (-0.4635,0.7005) -- (-0.4966,0.6081) -- (-0.4986,0.5100) -- (-0.4842,0.4129) -- (-0.4603,0.3178) -- (-0.4184,0.2291) -- (-0.3525,0.1564) -- (-0.2736,0.0979) -- (-0.1895,0.0475) -- (-0.0971,0.0144) -- cycle;
			\draw[red,thick] (0,0) -- (0,1);
			\draw[blue,thick] (0,0) -- (0.0980,0.9952) -- (-0.0971,0.0144) -- (0.1932,0.9713) -- (-0.1895,0.0475) -- (0.2819,0.9294) -- (-0.3525,0.1564) -- (0.3546,0.8635) -- (-0.4184,0.2291) -- (0.4635,0.7005) -- (-0.4603,0.3178) -- (0.4966,0.6081) -- (-0.4986,0.5100);
			\draw[blue,thick] (0,0) -- (-0.0980,0.9952) -- (0.0971,0.0144) -- (-0.1932,0.9713) -- (0.1895,0.0475) -- (-0.2819,0.9294) -- (0.3525,0.1564) -- (-0.3546,0.8635) -- (0.4184,0.2291) -- (-0.4635,0.7005) -- (0.4603,0.3178) -- (-0.4966,0.6081) -- (0.4986,0.5100);
			\draw (0.2819,0.9294) -- (-0.2736,0.0979);\draw (-0.2819,0.9294) -- (0.2736,0.0979);
			\draw (-0.4184,0.2291) -- (0.4131,0.7847);\draw (0.4184,0.2291) -- (-0.4131,0.7847);
			\draw (0.4966,0.6081) -- (-0.4842,0.4129);\draw (-0.4966,0.6081) -- (0.4842,0.4129);
		\end{tikzpicture}
	}
	\caption{Mossinghoff polygons $(\geo{M}_n,L(\geo{M}_n))$: (a) Hexadecagon~$\geo{M}_{16}$; (b) Triacontadigon~$\geo{M}_{32}$}
	\label{figure:Mn}
\end{figure}
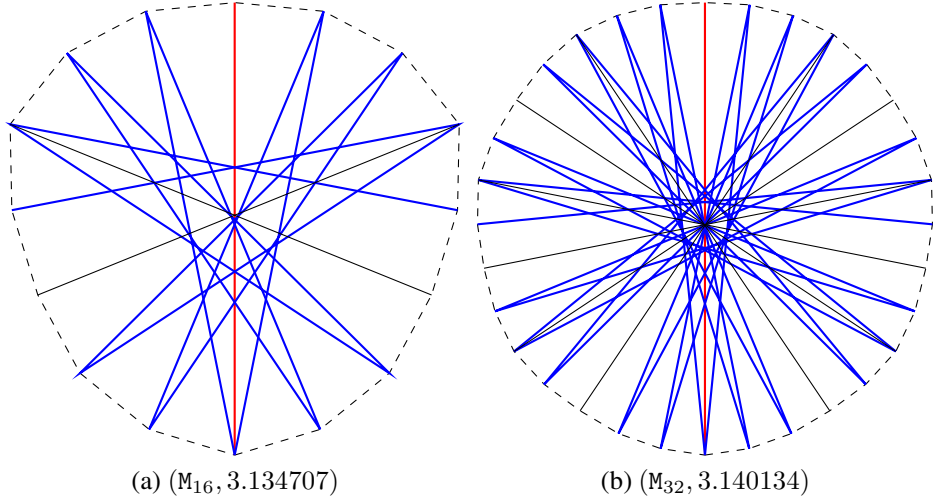

\section{Proof of Theorem~\ref{thm:Bn}}\label{sec:Bn}
Cartesian coordinates are used to describe an $n$-gon $\geo{P}_n$: a vertex $\geo{v}_i$, $i=0,1,\ldots,n-1$, is positioned at abscissa $x_i$ and ordinate $y_i$. Sums or differences of the indices of the coordinates are taken modulo~$n$. Placing the vertex $\geo{v}_0$ at the origin, we set $x_0 = y_0 = 0$. We also assume that the $n$-gon $\geo{P}_n$ is in the half-plane $y\ge 0$ and the vertices $\geo{v}_i$, $i=1,2,\ldots,n-1$, are arranged in a counterclockwise order as illustrated in Figure~\ref{figure:model}, i.e., $x_iy_{i+1} \ge y_ix_{i+1}$ for all $i=1,2,\ldots,n-2$.

The $n$-gon $\geo{P}_n$ is small if $\max_{i,j} \|\geo{v}_i - \geo{v}_j\| = 1$. It is equilateral if $\|\geo{v}_i - \geo{v}_{i-1}\| = c$ for all $i=1,2,\ldots,n$. Imposing that the determinants of the $2\times 2$ matrices satisfy
\[
\sigma_i :=
\begin{vmatrix}
	x_i - x_{i-1} & x_{i+1} - x_{i-1}\\
	y_i - y_{i-1} & y_{i+1} - y_{i-1}
\end{vmatrix}
\ge 0
\]
for all $i=1,2,\ldots,n-1$ ensures the convexity of the $n$-gon.

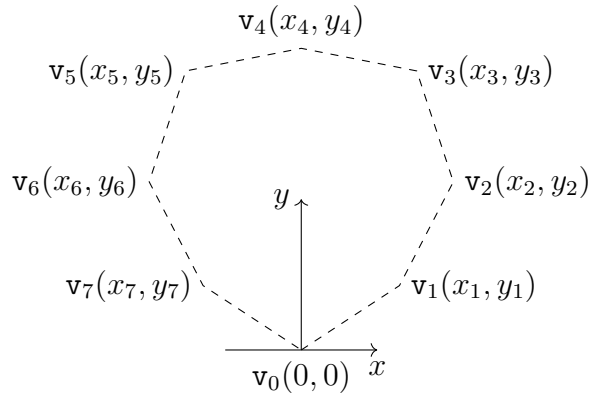
\begin{figure}[h]
	\centering
	\begin{tikzpicture}[scale=4]
		\draw[dashed] (0,0) node[below]{$\geo{v}_0(0,0)$} -- (0.3228,0.2134) node[right]{$\geo{v}_1(x_1,y_1)$} -- (0.5000,0.5574) node[right]{$\geo{v}_2(x_2,y_2)$} -- (0.3796,0.9251) node[right]{$\geo{v}_3(x_3,y_3)$} -- (0,1) node[above]{$\geo{v}_4(x_4,y_4)$} -- (-0.3796,0.9251) node[left]{$\geo{v}_5(x_5,y_5)$} -- (-0.5000,0.5574) node[left]{$\geo{v}_6(x_6,y_6)$} -- (-0.3228,0.2134) node[left]{$\geo{v}_7(x_7,y_7)$} -- cycle;
		\draw[->] (-0.25,0)--(0.25,0)node[below]{$x$};
		\draw[->] (0,0)--(0,0.5)node[left]{$y$};
	\end{tikzpicture}
	\caption{Definition of variables: Case of $n=8$ vertices}
	\label{figure:model}
\end{figure}

For any $n=2^s$ where $s\ge 4$ is an integer, we introduce a convex equilateral small $n$-gon called~$\geo{A}_n$ and constructed as follows. Its diameter graph has the edge $\geo{v}_0-\geo{v}_{\frac{n}{2}}$ as axis of symmetry and can be described by  the $(3n/8-1)$-length half-path $\geo{v}_0 - \geo{v}_{\frac{n}{2}-1} - \ldots - \geo{v}_{\frac{3n}{4}+1} - \geo{v}_{\frac{n}{4}}$ and the pendant edges $\geo{v}_0 - \geo{v}_{\frac{n}{2}}$, $\geo{v}_{4k-1} - \geo{v}_{4k-1+\frac{n}{2}}$, $k=1,2,\ldots,n/8$. The polygons $\geo{A}_{16}$ and $\geo{A}_{32}$ are shown  in Figure~\ref{figure:Bn}. They are symmetrical with respect to the vertical diameter.

\begin{figure}[h]
	\centering
	\subfloat[$(\geo{A}_{16},3.135288)$]{
		\begin{tikzpicture}[scale=6]
			\draw[dashed] (0,0) -- (0.1875,0.0569) -- (0.3604,0.1492) -- (0.4847,0.3006) -- (0.4960,0.4963) -- (0.4390,0.6838) -- (0.3465,0.8565) -- (0.1950,0.9808) -- (0,1) -- (-0.1950,0.9808) -- (-0.3465,0.8565) -- (-0.4390,0.6838) -- (-0.4960,0.4963) -- (-0.4847,0.3006) -- (-0.3604,0.1492) -- (-0.1875,0.0569) -- cycle;
			\draw[red,thick] (0,0)--(0,1);
			\draw[blue,thick] (0,0) -- (0.1950,0.9808) -- (-0.3604,0.1492) -- (0.3465,0.8565) -- (-0.4847,0.3006) -- (0.4960,0.4963);
			\draw[blue,thick] (0,0) -- (-0.1950,0.9808) -- (0.3604,0.1492) -- (-0.3465,0.8565) -- (0.4847,0.3006) -- (-0.4960,0.4963);
			\draw (0.1950,0.9808) -- (-0.1875,0.0569);\draw (-0.1950,0.9808) -- (0.1875,0.0569);
			\draw (-0.4847,0.3006) -- (0.4390,0.6838);\draw (0.4847,0.3006) -- (-0.4390,0.6838);
		\end{tikzpicture}
	}
	\subfloat[$(\geo{A}_{32},3.140246)$]{
		\begin{tikzpicture}[scale=6]
			\draw[dashed] (0,0) -- (0.0971,0.0144) -- (0.1923,0.0382) -- (0.2810,0.0802) -- (0.3537,0.1461) -- (0.4121,0.2249) -- (0.4626,0.3091) -- (0.4957,0.4015) -- (0.4995,0.4995) -- (0.4851,0.5966) -- (0.4613,0.6918) -- (0.4193,0.7805) -- (0.3534,0.8532) -- (0.2746,0.9117) -- (0.1904,0.9621) -- (0.0980,0.9952) -- (0,1) -- (-0.0980,0.9952) -- (-0.1904,0.9621) -- (-0.2746,0.9117) -- (-0.3534,0.8532) -- (-0.4193,0.7805) -- (-0.4613,0.6918) -- (-0.4851,0.5966) -- (-0.4995,0.4995) -- (-0.4957,0.4015) -- (-0.4626,0.3091) -- (-0.4121,0.2249) -- (-0.3537,0.1461) -- (-0.2810,0.0802) -- (-0.1923,0.0382) -- (-0.0971,0.0144) -- cycle;
			\draw[red,thick] (0,0)--(0,1);
			\draw[blue,thick] (0,0) -- (0.0980,0.9952) -- (-0.1923,0.0382) -- (0.1904,0.9621) -- (-0.2810,0.0802) -- (0.3534,0.8532) -- (-0.3537,0.1461) -- (0.4193,0.7805) -- (-0.4626,0.3091) -- (0.4613,0.6918) -- (-0.4957,0.4015) -- (0.4995,0.4995);
			\draw[blue,thick] (0,0) -- (-0.0980,0.9952) -- (0.1923,0.0382) -- (-0.1904,0.9621) -- (0.2810,0.0802) -- (-0.3534,0.8532) -- (0.3537,0.1461) -- (-0.4193,0.7805) -- (0.4626,0.3091) -- (-0.4613,0.6918) -- (0.4957,0.4015) -- (-0.4995,0.4995);
			\draw (0.0980,0.9952) -- (-0.0971,0.0144);\draw (-0.0980,0.9952) -- (0.0971,0.0144);
			\draw (-0.2810,0.0802) -- (0.2746,0.9117);\draw (0.2810,0.0802) -- (-0.2746,0.9117);
			\draw (0.4193,0.7805) -- (-0.4121,0.2249);\draw (-0.4193,0.7805) -- (0.4121,0.2249);
			\draw (-0.4957,0.4015) -- (0.4851,0.5966);\draw (0.4957,0.4015) -- (-0.4851,0.5966);
		\end{tikzpicture}
	}
	\caption{Polygons $(\geo{A}_n,L(\geo{A}_n))$ defined in Theorem~\ref{thm:Bn}: (a) Hexadecagon $\geo{A}_{16}$; (b) Triacontadigon $\geo{A}_{32}$}
	\label{figure:Bn}
\end{figure}
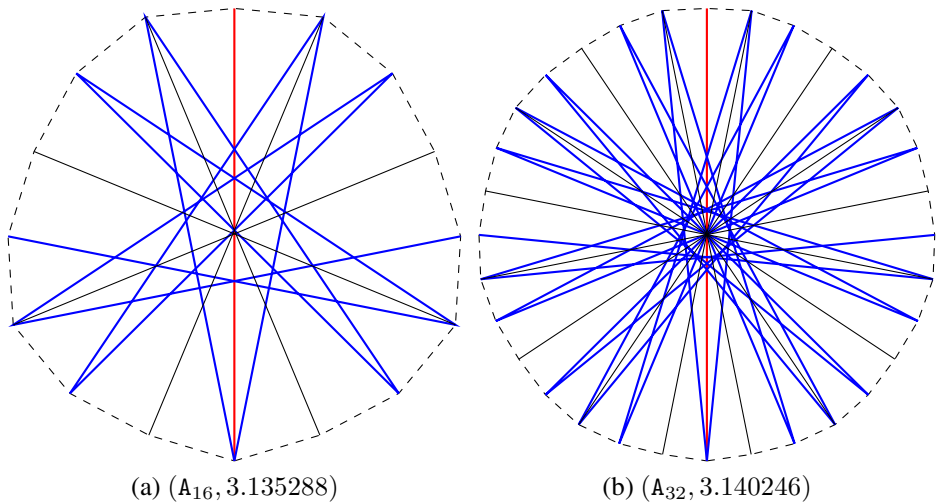

Place the vertex $\geo{v}_{\frac{n}{2}}$ at $(0,1)$ in the plane. Let $t = \angle \geo{v}_{\frac{n}{2}} \geo{v}_0 \geo{v}_{\frac{n}{2}-1} \in (0,\pi/n)$. This implies that the sides length of~$\geo{A}_{n}$ is $2\sin (t/2)$. Since $\geo{A}_n$ is equilateral and symmetric, we have from the half-path $\geo{v}_0 - \ldots - \geo{v}_{\frac{n}{4}}$,
\[
\begin{aligned}
	x_{\frac{3n}{4}+1} &= \sin t - \sum_{k=1}^{n/8-1} (-1)^{k-1} (\sin (4k-1)t - \sin 4kt + \sin (4k+1)t)\\
	&= \sin t - \frac{(2\cos t -1)(\sin 2t + \sin (n/2-2)t)}{2\cos 2t} &&= -x_{\frac{n}{4}-1},\\
	x_{\frac{n}{4}} &= x_{\frac{3n}{4}+1} + \sin (n/2-1)t &&= -x_{\frac{3n}{4}},\\
	y_{\frac{3n}{4}+1} &= \cos t - \sum_{k=1}^{n/8-1} (-1)^{k-1} (\cos (4k-1)t - \cos 4kt + \cos (4k+1)t)\\
	&= \cos t - \frac{(2\cos t -1)(\cos 2t + \cos (n/2-2)t)}{2\cos 2t} &&= y_{\frac{n}{4}-1},\\
	y_{\frac{n}{4}} &= y_{\frac{3n}{4}+1} + \cos (n/2-1)t &&= y_{\frac{3n}{4}}.
\end{aligned}
\]
	
Finally, the angle $t$ is chosen so that $\|\geo{v}_{\frac{3n}{4}+1}-\geo{v}_{\frac{3n}{4}}\| = 2\sin (t/2)$, i.e.,
\[
(2x_{\frac{3n}{4}+1} + \sin (n/2-1)t)^2+\cos^2 (n/2-1)t = 4\sin^2(t/2).
\]
An asymptotic analysis produces that, for large $n$, this equation has a solution $t_0(n)$ satisfying
\[
t_0(n) = \frac{\pi}{n} - \frac{\pi^4}{n^5} + \frac{\pi^5}{n^6} -\frac{11\pi^6}{6n^7} + \frac{35\pi^7}{12n^8} + O\left(\frac{1}{n^9}\right).
\]
By setting $t = t_0(n)$, the perimeter of $\geo{A}_n$ is
\[
\begin{aligned}
	L(\geo{A}_n) &= 2n\sin \frac{t_0(n)}{2} = 2n \sin \left(\frac{\pi}{2n} - \frac{\pi^4}{2n^5} + O\left(\frac{1}{n^6}\right)\right)\\
	&= \pi - \frac{\pi^3}{24n^2} + \left(\frac{\pi^5}{1920}-\pi^4\right)\frac{1}{n^4} + \frac{\pi^5}{n^5} - \left(\frac{\pi^7}{322560}+\frac{41\pi^6}{24}\right)\frac{1}{n^6} + O\left(\frac{1}{n^7}\right)
\end{aligned}
\]
and
\[
\ub{L}_n - L(\geo{A}_n) = \frac{\pi^4}{n^4} - \frac{\pi^5}{n^5} + O\left(\frac{1}{n^6}\right).
\]
Since the polygon $\geo{M}_n$ proposed by Mossinghoff~\cite{mossinghoff2008} satisfies
\[
L(\geo{M}_n) = \pi - \frac{\pi^3}{24n^2} + \left(\frac{\pi^5}{1920}-3\pi^4\right)\frac{1}{n^4} + \frac{9\pi^5}{n^5} - \left(\frac{\pi^7}{322560}+\frac{9\pi^6}{8}\right)\frac{1}{n^6} + O\left(\frac{1}{n^7}\right),
\]
it follows that
\[
L(\geo{A}_n) - L(\geo{M}_n) = \frac{2\pi^4}{n^4} - \frac{8\pi^5}{n^5} - \frac{7\pi^6}{12n^6} + O\left(\frac{1}{n^7}\right).
\]

To verify that $\geo{A}_n$ is small, we calculate
\[
\|\geo{v}_{\frac{n}{4}} - \geo{v}_{\frac{3n}{4}}\| = 2x_{\frac{n}{4}} = 1 - \frac{\pi^3}{n^3} - \frac{7\pi^5}{4n^5} + O\left(\frac{1}{n^7}\right) < 1.
\]
To test that $\geo{A}_n$ is convex, we compute
\[
\sigma_{\frac{n}{4}} = \frac{2\pi^3}{n^3} - \frac{\pi^4}{n^4} + O\left(\frac{1}{n^5}\right) > 0.
\]
This completes the proof of Theorem~\ref{thm:Bn}.\qed

All polygons presented in this work and in~\cite{bingane2022a,bingane2022b} were implemented as a package: OPTIGON~\cite{optigon}, which is freely available on GitHub. In OPTIGON, we provide Julia and MATLAB functions that give the coordinates of the vertices. One can also find an algorithm developed in~\cite{bingane2022d} to find an estimate of the maximal area of a small $n$-gon when $n \ge 6$ is even.

Table~\ref{table:perimeter} shows the perimeters of $\geo{A}_n$, along with the upper bounds $\ub{L}_n$, the perimeters of the regular polygons $\geo{R}_n$ and Mossinghoff polygons $\geo{M}_n$. When $n = 2^s$ and $s\ge 4$, $\geo{A}_n$ provides a tighter lower bound on the maximal perimeter $\ell_n^*$ compared to the best prior convex equilateral small $n$-gon~$\geo{M}_n$. As $n$ increases, it is not surprising that the fraction $\frac{L(\geo{A}_n) - L(\geo{M}_n)}{\ub{L}_n - L(\geo{M}_n)}$ of the length of the interval $[L(\geo{M}_n), \ub{L}_n]$ containing $L(\geo{A}_n)$ approaches $\frac{2}{3}$ since $L(\geo{A}_n) - L(\geo{M}_n) \sim \frac{2\pi^4}{n^4}$ and $\ub{L}_n - L(\geo{M}_n) \sim \frac{3\pi^4}{n^4}$ for large~$n$.

\begin{table}[h]
	\footnotesize
	\centering
	\caption{Perimeters of $\geo{A}_n$}
	\label{table:perimeter}
		\begin{tabular}{@{}rllllr@{}}
			\toprule
			$n$ & $L(\geo{R}_n)$ & $L(\geo{M}_n)$ & $L(\geo{A}_n)$ & $\ub{L}_n$ & $\frac{L(\geo{A}_n) - L(\geo{M}_n)}{\ub{L}_n - L(\geo{M}_n)}$ \\
			\midrule
			16	&	3.1214451523	&	3.1347065475	&	3.1352878881	&	3.1365484905	&	0.3156	\\
			32	&	3.1365484905	&	3.1401338091	&	3.1402460942	&	3.1403311570	&	0.5690	\\
			64	&	3.1403311570	&	3.1412623836	&	3.1412717079	&	3.1412772509	&	0.6272	\\
			128	&	3.1412772509	&	3.1415127924	&	3.1415134468	&	3.1415138011	&	0.6487	\\
			256	&	3.1415138011	&	3.1415728748	&	3.1415729180	&	3.1415729404	&	0.6589	\\
			\bottomrule
		\end{tabular}
\end{table}

\section{Improved triacontadigon and hexacontatetragon}\label{sec:optimal}
It is natural to ask if the polygon constructed $\geo{A}_n$ might be optimal for some $n$. Using constructive arguments, Proposition~\ref{thm:Z32} and Proposition~\ref{thm:Z64} show that $\geo{A}_{32}$ and $\geo{A}_{64}$ are suboptimal.

\begin{proposition}\label{thm:Z32}
There exists a convex equilateral small $32$-gon whose perimeter exceeds that of~$\geo{A}_{32}$.
\end{proposition}
\begin{proof}
Consider the $32$-gon $\geo{Z}_{32}$, illustrated in Figure~\ref{figure:Z32}. Its diameter graph has the edge $\geo{v}_0-\geo{v}_{16}$ as axis of symmetry and can be described by the $4$-length half-path $\geo{v}_0-\geo{v}_{11}-\geo{v}_{24}-\geo{v}_{10}-\geo{v}_{23}$ and the pendant edges $\geo{v}_{0} - \geo{v}_{15}, \ldots, \geo{v}_{0} - \geo{v}_{12}$, $\geo{v}_{11} - \geo{v}_{31}, \ldots, \geo{v}_{11} - \geo{v}_{25}$.
	
Place the vertex $\geo{v}_0$ at $(0,0)$ in the plane, and the vertex $\geo{v}_{16}$ at $(0,1)$. Let $t = \angle \geo{v}_{16} \geo{v}_0 \geo{v}_{15} \in (0,\pi/32)$. We have, from the half-path $\geo{v}_0 - \ldots -\geo{v}_{23}$,
	\[
	\begin{aligned}
		x_{10} &= \sin 5t - \sin 13t + \sin 14t &&= -x_{22}, & y_{10} &= \cos 5t - \cos 13t + \cos 14t &&= y_{11},\\
		x_{23} &= x_{10} - \sin 15t &&= -x_9, & y_{23} &= y_{10} - \cos 15t &&= y_9.
	\end{aligned}
	\]

Finally, $t$ is chosen so that $\|\geo{v}_{10}-\geo{v}_9\| = 2\sin (t/2)$, i.e.,
\[
(2(\sin 5t - \sin 13t + \sin 14t) - \sin 15t)^2+\cos^2 15t = 4\sin^2(t/2).
\]
We obtain $t = 0.0981744286\ldots$ and $L(\geo{Z}_{32}) = 64\sin (t/2) = 3.1403202339\ldots  > L(\geo{A}_{32})$. One can verify that $\geo{Z}_{32}$ is small and convex.
\end{proof}

\begin{proposition}\label{thm:Z64}
There exists a convex equilateral small $64$-gon whose perimeter exceeds that of~$\geo{A}_{64}$.
\end{proposition}
\begin{proof}
Consider the $64$-gon $\geo{Z}_{64}$, illustrated in Figure~\ref{figure:Z64}. Its diameter graph has the edge $\geo{v}_0-\geo{v}_{32}$ as axis of symmetry and can be described by the $23$-length half-path $\geo{v}_0-\geo{v}_{31}-\geo{v}_{63}-\geo{v}_{30}-\geo{v}_{61} - \geo{v}_{29} - \geo{v}_{60} - \geo{v}_{28} - \geo{v}_{58} - \geo{v}_{27} - \geo{v}_{57} - \geo{v}_{26} - \geo{v}_{56} - \geo{v}_{25} - \geo{v}_{55} - \geo{v}_{24} - \geo{v}_{54} - \geo{v}_{23} - \geo{v}_{53} - \geo{v}_{21} - \geo{v}_{52} - \geo{v}_{19} - \geo{v}_{51} - \geo{v}_{16}$, the pendant edges $\geo{v}_{30}-\geo{v}_{62}$, $\geo{v}_{28} - \geo{v}_{59}$, $\geo{v}_{53} - \geo{v}_{22}$, $\geo{v}_{52} - \geo{v}_{20}$, $\geo{v}_{51} - \geo{v}_{18}$, $\geo{v}_{51}-\geo{v}_{17}$, and the $4$-length path $\geo{v}_{15}-\geo{v}_{50} - \geo{v}_{14} - \geo{v}_{49}$.

Place the vertex $\geo{v}_0$ at $(0,0)$ in the plane, and the vertex $\geo{v}_{32}$ at $(0,1)$. Let $t = \angle \geo{v}_{32} \geo{v}_0 \geo{v}_{31} \in (0,\pi/64)$. We have, from the half-path $\geo{v}_0 - \ldots -\geo{v}_{31}$,
\[
\begin{aligned}
	x_{51} &= \sin t - \sin 2t + \sin 3t - \sin 5t + \sin 6t -\sin 7t + \sin 8t\\
	&-\sum_{k=10}^{20} (-1)^k \sin kt + \sin 22t - \sin 23t + \sin 25t - \sin 26t &&= -x_{13},\\
	y_{51} &= \cos t - \cos 2t + \cos 3t - \cos 5t + \cos 6t -\cos 7t + \cos 8t\\
	&-\sum_{k=10}^{20} (-1)^k \cos kt + \cos 22t - \cos 23t + \cos 25t - \cos 26t &&= y_{13},\\
	x_{16} &= x_{51} + \sin 29t &&= -x_{48},\\
	y_{16} &= y_{51} + \cos 29t &&= y_{48},
\end{aligned}
\]
and, from the path $\geo{v}_{15} - \ldots -\geo{v}_{49}$,
\[
\begin{aligned}
	x_{50} &= -1/2 &&= -x_{14}, & y_{50} &= y &&= y_{14},\\
	x_{15} &= x_{50} + \cos t &&= -x_{49}, & y_{15} &= y_{50} + \sin t &&= y_{49}.
\end{aligned}
\]

Finally, $t$ and $y$ are chosen so that $\|\geo{v}_{51}-\geo{v}_{50}\| = \|\geo{v}_{16}-\geo{v}_{15}\| = 2\sin (t/2)$. We obtain $t = 0.0490873533\ldots$ and $L(\geo{Z}_{64}) = 128 \sin(t/2)= 3.1412752155\ldots  > L(\geo{A}_{64})$. One can verify that $\geo{Z}_{64}$ is small and convex.
\end{proof}

Polygons $\geo{Z}_{32}$ and $\geo{Z}_{64}$ offer a significant improvement to the lower bound of the optimal value. We note that
\[
\begin{aligned}
	\ell_{32}^* - L(\geo{Z}_{32}) &< \ub{L}_{32} - L(\geo{Z}_{32}) = 1.09\ldots \times 10^{-5} < \ub{L}_{32} - L(\geo{A}_{32}) = 8.50\ldots \times 10^{-5},\\
	\ell_{64}^* - L(\geo{Z}_{64}) &< \ub{L}_{64} - L(\geo{Z}_{64}) = 2.03\ldots \times 10^{-6} < \ub{L}_{64} - L(\geo{A}_{64}) = 5.54\ldots \times 10^{-6}.
\end{aligned}
\]
Also, the fractions
\[
\begin{aligned}
	\frac{L(\geo{Z}_{32}) - L(\geo{A}_{32})}{\ub{L}_{32} - L(\geo{A}_{32})} &= 0.8715\ldots,\\
	\frac{L(\geo{Z}_{64}) - L(\geo{A}_{64})}{\ub{L}_{64} - L(\geo{A}_{64})} &= 0.6327\ldots
\end{aligned}
\]
indicate that the perimeters of the improved polygons are quite close to the maximal perimeter. This suggests that it is possible that another family of convex equilateral small polygons might produce an improvement to Theorem~\ref{thm:Bn}.

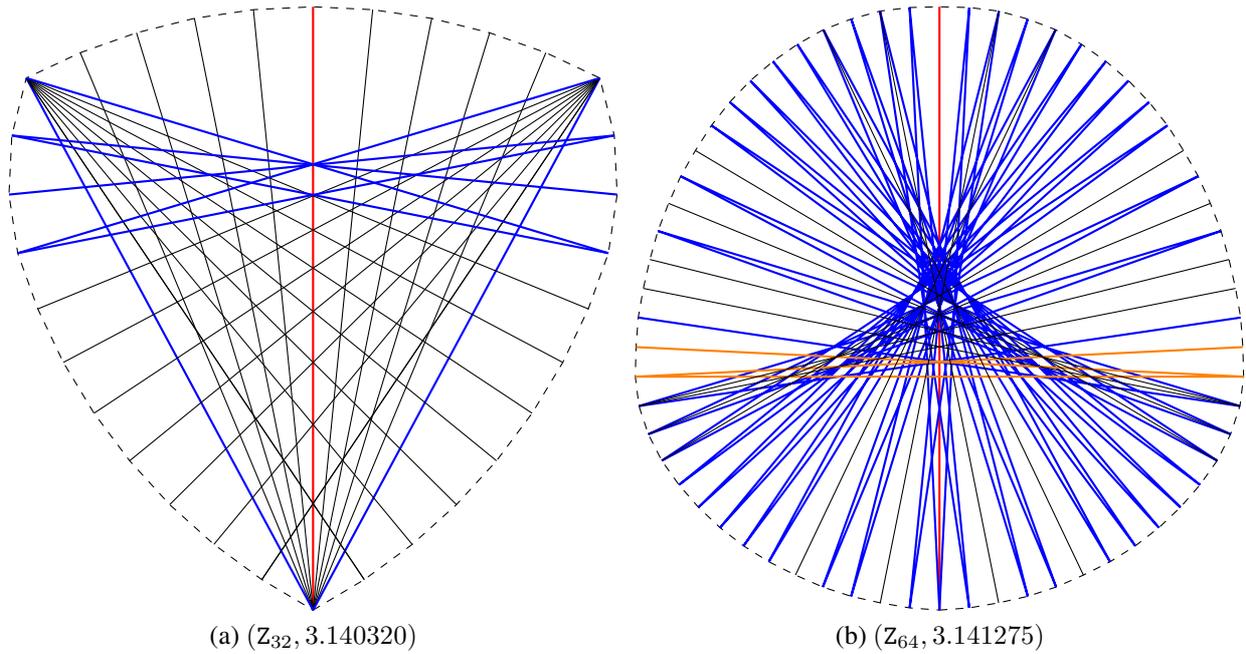
\begin{figure}[h]
\centering
\subfloat[$(\geo{Z}_{32},3.140320)$]{
\begin{tikzpicture}[scale=8]
\draw[dashed] (0,0) -- (0.0842,0.0505) -- (0.1630,0.1089) -- (0.2357,0.1748) -- (0.3016,0.2475) -- (0.3601,0.3263) -- (0.4105,0.4105) -- (0.4525,0.4992) -- (0.4855,0.5916) -- (0.4999,0.6887) -- (0.4952,0.7867) -- (0.4714,0.8819) -- (0.3827,0.9239) -- (0.2903,0.9569) -- (0.1951,0.9808) -- (0.0980,0.9952) -- (0,1) -- (-0.0980,0.9952) -- (-0.1951,0.9808) -- (-0.2903,0.9569) -- (-0.3827,0.9239) -- (-0.4714,0.8819) -- (-0.4952,0.7867) -- (-0.4999,0.6887) -- (-0.4855,0.5916) -- (-0.4525,0.4992) -- (-0.4105,0.4105) -- (-0.3601,0.3263) -- (-0.3016,0.2475) -- (-0.2357,0.1748) -- (-0.1630,0.1089) -- (-0.0842,0.0505) -- cycle;
\draw[red,thick] (0,0)--(0,1);
\draw (0,0)--(0.0980,0.9952);\draw (0,0)--(-0.0980,0.9952);
\draw (0,0)--(0.1951,0.9808);\draw (0,0)--(-0.1951,0.9808);
\draw (0,0)--(0.2903,0.9569);\draw (0,0)--(-0.2903,0.9569);
\draw (0,0)--(0.3827,0.9239);\draw (0,0)--(-0.3827,0.9239);
\draw[blue,thick] (0,0)--(0.4714,0.8819);\draw[blue,thick] (0,0)--(-0.4714,0.8819);
\draw (0.0842,0.0505)--(-0.4714,0.8819);\draw (-0.0842,0.0505)--(0.4714,0.8819);
\draw (0.1630,0.1089)--(-0.4714,0.8819);\draw (-0.1630,0.1089)--(0.4714,0.8819);
\draw (0.0842,0.0505)--(-0.4714,0.8819);\draw (-0.0842,0.0505)--(0.4714,0.8819);
\draw (0.2357,0.1748)--(-0.4714,0.8819);\draw (-0.2357,0.1748)--(0.4714,0.8819);
\draw (0.3016,0.2475)--(-0.4714,0.8819);\draw (-0.3016,0.2475)--(0.4714,0.8819);
\draw (0.3601,0.3263)--(-0.4714,0.8819);\draw (-0.3601,0.3263)--(0.4714,0.8819);
\draw (0.4105,0.4105)--(-0.4714,0.8819);\draw (-0.4105,0.4105)--(0.4714,0.8819);
\draw (0.4525,0.4992)--(-0.4714,0.8819);\draw (-0.4525,0.4992)--(0.4714,0.8819);
\draw[blue,thick] (0.4855,0.5916)--(-0.4714,0.8819);\draw[blue,thick] (-0.4855,0.5916)--(0.4714,0.8819);
\draw[blue,thick] (0.4952,0.7867)--(-0.4855,0.5916);\draw[blue,thick] (-0.4952,0.7867)--(0.4855,0.5916);
\draw[blue,thick] (0.4952,0.7867)--(-0.4999,0.6887);\draw[blue,thick] (-0.4952,0.7867)--(0.4999,0.6887);
\end{tikzpicture}
\label{figure:Z32}
}
\subfloat[$(\geo{Z}_{64},3.141275)$]{
\begin{tikzpicture}[scale=8]
\draw[dashed] (0,0) -- (0.0490,0.0036) -- (0.0973,0.0120) -- (0.1452,0.0228) -- (0.1918,0.0382) -- (0.2367,0.0580) -- (0.2805,0.0801) -- (0.3220,0.1064) -- (0.3607,0.1366) -- (0.3962,0.1704) -- (0.4283,0.2076) -- (0.4565,0.2477) -- (0.4786,0.2915) -- (0.4940,0.3382) -- (0.5000,0.3869) -- (0.4988,0.4359) -- (0.4952,0.4849) -- (0.4868,0.5332) -- (0.4760,0.5811) -- (0.4629,0.6284) -- (0.4453,0.6742) -- (0.4254,0.7191) -- (0.4012,0.7618) -- (0.3749,0.8033) -- (0.3447,0.8420) -- (0.3109,0.8775) -- (0.2737,0.9096) -- (0.2336,0.9378) -- (0.1909,0.9620) -- (0.1451,0.9797) -- (0.0978,0.9928) -- (0.0491,0.9988) -- (0,1) -- (-0.0491,0.9988) -- (-0.0978,0.9928) -- (-0.1451,0.9797) -- (-0.1909,0.9620) -- (-0.2336,0.9378) -- (-0.2737,0.9096) -- (-0.3109,0.8775) -- (-0.3447,0.8420) -- (-0.3749,0.8033) -- (-0.4012,0.7618) -- (-0.4254,0.7191) -- (-0.4453,0.6742) -- (-0.4629,0.6284) -- (-0.4760,0.5811) -- (-0.4868,0.5332) -- (-0.4952,0.4849) -- (-0.4988,0.4359) -- (-0.5000,0.3869) -- (-0.4940,0.3382) -- (-0.4786,0.2915) -- (-0.4565,0.2477) -- (-0.4283,0.2076) -- (-0.3962,0.1704) -- (-0.3607,0.1366) -- (-0.3220,0.1064) -- (-0.2805,0.0801) -- (-0.2367,0.0580) -- (-0.1918,0.0382) -- (-0.1452,0.0228) -- (-0.0973,0.0120) -- (-0.0490,0.0036) -- cycle;
\draw[red,thick] (0,0) -- (0,1);
\draw[blue,thick] (0,0) -- (0.0491,0.9988) -- (-0.0490,0.0036) -- (0.0978,0.9928) -- (-0.1452,0.0228) -- (0.1451,0.9797) -- (-0.1918,0.0382) -- (0.1909,0.9620) -- (-0.2805,0.0801) -- (0.2336,0.9378) -- (-0.3220,0.1064) -- (0.2737,0.9096) -- (-0.3607,0.1366) -- (0.3109,0.8775) -- (-0.3962,0.1704) -- (0.3447,0.8420) -- (-0.4283,0.2076) -- (0.3749,0.8033) -- (-0.4565,0.2477) -- (0.4254,0.7191) -- (-0.4786,0.2915) -- (0.4629,0.6284) -- (-0.4940,0.3382) -- (0.4952,0.4849);
\draw[blue,thick] (0,0) -- (-0.0491,0.9988) -- (0.0490,0.0036) -- (-0.0978,0.9928) -- (0.1452,0.0228) -- (-0.1451,0.9797) -- (0.1918,0.0382) -- (-0.1909,0.9620) -- (0.2805,0.0801) -- (-0.2336,0.9378) -- (0.3220,0.1064) -- (-0.2737,0.9096) -- (0.3607,0.1366) -- (-0.3109,0.8775) -- (0.3962,0.1704) -- (-0.3447,0.8420) -- (0.4283,0.2076) -- (-0.3749,0.8033) -- (0.4565,0.2477) -- (-0.4254,0.7191) -- (0.4786,0.2915) -- (-0.4629,0.6284) -- (0.4940,0.3382) -- (-0.4952,0.4849);
\draw[orange,thick] (0.4988,0.4359) -- (-0.5000,0.3869) -- (0.5000,0.3869) -- (-0.4988,0.4359);
\draw (0.0978,0.9928) -- (-0.0973,0.0120);\draw (-0.0978,0.9928) -- (0.0973,0.0120);
\draw (0.1909,0.9620) -- (-0.2367,0.0580);\draw (-0.1909,0.9620) -- (0.2367,0.0580);
\draw (-0.4565,0.2477) -- (0.4012,0.7618);\draw (0.4565,0.2477) -- (-0.4012,0.7618);
\draw (-0.4786,0.2915) -- (0.4453,0.6742);\draw (0.4786,0.2915) -- (-0.4453,0.6742);
\draw (-0.4940,0.3382) -- (0.4760,0.5811);\draw (0.4940,0.3382) -- (-0.4760,0.5811);
\draw (-0.4940,0.3382) -- (0.4868,0.5332);\draw (0.4940,0.3382) -- (-0.4868,0.5332);
\end{tikzpicture}
\label{figure:Z64}
}
\caption{Improved convex equilateral small $n$-gons $(\geo{Z}_n,L(\geo{Z}_n))$: (a) Triacontadigon $\geo{Z}_{32}$ with larger perimeter than $\geo{A}_{32}$; (b) Hexacontatetragon $\geo{Z}_{64}$ with larger perimeter than $\geo{A}_{64}$}
\end{figure}

\section{Conclusion}\label{sec:conclusion}
Lower bounds on the maximal perimeter of convex equilateral small $n$-gons were provided when $n$ is a power of $2$ and these bounds are tighter than the previous ones from the literature. For any $n=2^s$ with integer $s\ge 4$, we constructed a convex equilateral small $n$-gon $\geo{A}_n$ whose perimeter is within $\pi^4/n^4 + O(1/n^5)$ of the optimal value. For $n=32$ and $n=64$, we propose solutions with even larger perimeters.

\bibliographystyle{ieeetr}
\bibliography{../../research}

\end{document}